\documentclass[11pt]{article}
    
\usepackage{amsmath,amsfonts,amsthm,amscd,amssymb,graphicx}

\usepackage{subfigure}

\usepackage{hyperref}
\usepackage{color}

\numberwithin{equation}{section}
\newtheorem{theorem}{Theorem}[section]

\newtheorem{lemma}[theorem]{Lemma}

\newtheorem{proposition}[theorem]{Proposition}
\newtheorem{prop}[theorem]{Proposition}

\def\eps{\varepsilon }

\newcommand{\RR}{\mathbb{R}}

\newcommand{\CC}{\mathbb{C}}

\newcommand{\NN}{{\mathbb N}}
\newcommand{\ZZ}{{\mathbb Z}}
\newcommand{\TT}{{\mathbb T}}
\def\beq{\begin{equation}}
\def\eeq{\end{equation}}
\def\bb1{{1\!\!1}}
\def\cB{\mathcal{B}}
\def\cA{\mathcal{A}}

\def\OS{\mathrm{OS}}

\def\rit{{\Bbb R}}

\def\eps{\varepsilon}

\def\bl{\mathrm{bl}}

\begin{document}

%\title{Sharp bounds on the linear semigroup of Navier Stokes with boundary layer norms} 
\title{Sharp bounds for the resolvent of linearized Navier Stokes equations in the half space around a  shear profile} 

\author{Emmanuel Grenier\footnotemark[1]
  \and Toan T. Nguyen\footnotemark[2]
}

\maketitle

\renewcommand{\thefootnote}{\fnsymbol{footnote}}

\footnotetext[1]{Equipe Projet Inria NUMED,
 INRIA Rh\^one Alpes, Unit\'e de Math\'ematiques Pures et Appliqu\'ees., 
 UMR 5669, CNRS et \'Ecole Normale Sup\'erieure de Lyon,
               46, all\'ee d'Italie, 69364 Lyon Cedex 07, France. Email: Emmanuel.Grenier@ens-lyon.fr}

\footnotetext[2]{Department of Mathematics, Penn State University, State College, PA 16803. 
Email: nguyen@math.psu.edu. TN's research was partly supported by the NSF under grant DMS-1405728.}

%%%%%%%%

%\subsection*{Abstract}

%%%%%%%%

\begin{abstract}

In this paper, we derive sharp bounds on the semigroup of the linearized incompressible Navier-Stokes equations 
near a stationary shear layer in the half plane and in the half space ($\rit_+^2$ or $\rit_+^3$), with Dirichlet boundary
conditions, assuming that this shear layer in spectrally unstable for Euler equations.  
In the inviscid limit, due to the prescribed no-slip boundary conditions, 
vorticity becomes unbounded near the boundary. The novelty of this paper is to introduce boundary layer norms 
that capture the unbounded vorticity and to derive sharp estimates on this vorticity that are uniform in the inviscid limit. 

\end{abstract}

%\tableofcontents

%%%%%%%%%%%%%%%%%%%%%%

\section{Introduction}

%%%%%%%%%%%%%%%%%%%%%%

%%%%%%

\subsection{Position of the problem}

%%%%%%

In this paper, we study the linearized incompressible Navier Stokes equations on the half space $(x,z) \in \TT \times \RR_+$
around a stationary boundary layer profile of the form $U_\bl = (U(z),0)$, where
$U$ is a smooth function with $U(0) = 0$. In the whole paper, $x$ is periodic, with
unit period and $z > 0$. As will be apparent in the proof, the three dimensional case $\TT^2 \times \RR_+$ is similar.
Precisely, we consider the linear problem 
\beq \label{linearizedNS}
\begin{aligned}
\partial_t v + U_\bl \cdot \nabla v + v \cdot \nabla U_\bl + \nabla p &=  \nu \Delta v
\\
\nabla \cdot v & =0
\end{aligned}\eeq
in the half space $\TT \times \RR_+$, where $p$ is the scalar pressure 
and $v = (v_1,v_2)$ denotes the velocity vector field, satisfying the classical no-slip boundary condition
\begin{equation}\label{BC}
v_{\vert_{z=0}} =0.
\end{equation}
We will denote by $\omega$ the vorticity
$$
\omega= \partial_z v_1 - \partial_x v_2
$$
We are interested in the linear problem \eqref{linearizedNS} in the vanishing viscosity limit $\nu \to0$. 
The linearized problem around  a stationary profile $U_\bl$ is a classical problem in Fluid Mechanics
and arises in the study of boundary layer instabilities 
and of the onset of turbulence. It has attracted prominent physicists, including Rayleigh, 
Orr, Sommerfeld, Heisenberg, Tollmien, C.C. Lin, Schlichting, among others. 
For a review of the physical literature on the subject, we refer the readers to \cite{Reid}.

Two cases arise. Either the profile $U_{bl}$ is spectrally unstable for the underlying Euler equation.
In this case it is also spectrally unstable for Navier Stokes equations, with a most unstable eigenvalue with
a $O(1)$ real part. Or the profile is spectrally stable for the underlying Euler equation. In this case
it turns out that it is spectrally {\it} unstable for Navier Stokes equations, with a much slower instability,
namely with a real part of order $O(\nu^{1/2})$. In this paper we focus on the first case and prove
that the linear growth rate of a solution of linearized Navier Stokes equation is arbitrarily close
to the spectral radius of linearized Navier Stokes equations, which is known to be arbitrarily close
to that of Euler equations.

Motivated by the study of instabilities of Prandtl's boundary layers \cite{Grenier00CPAM,GGN1,GGN3},  
we are interested in deriving sharp bounds on the semigroup of \eqref{linearizedNS} that are uniform in the vanishing viscosity limit. 
As $\nu\to 0$, solutions to the linear problem \eqref{linearizedNS} are expected to converge to solutions 
to the corresponding linearized Euler problem around $U_\bl$. 
In the limit, however, due to the discrepancy of the corresponding boundary conditions between Euler and Navier-Stokes equations, 
boundary layers of thickness of order $\sqrt \nu$ appear, and thus, the vorticity near the boundary is of order $\nu^{-1/2}$
and is unbounded as $\nu$ vanishes.

In this paper, we introduce boundary layer function spaces that capture the behavior of the vorticity near the boundary. 
Our semigroup estimates are uniform in the inviscid limit. 

%%%%%%%%

\subsection{Boundary layer spaces}

%%%%%%%%

Let us first introduce various functional spaces.
Let $\beta,\gamma>0$ be fixed and let 
\begin{equation}\label{sub-thickness}
\delta = \gamma \sqrt \nu
\end{equation}
be the boundary layer thickness. The constant $\gamma$ will be fixed later, large enough.
For any function of one variable $f = f(z)$, we introduce the boundary layer norm 
\begin{equation}\label{def-norm}
\| f \|_{\beta,\gamma}  = \sup_{z \ge 0} | f(z) | e^{\beta  z} 
\Bigl( 1 + \delta^{-1} \phi_{P} (\delta^{-1} z)  \Bigr)^{-1}
\end{equation}
with boundary layer weight function 
$$ 
\phi_P(z) = \frac{1}{1+z^P} 
$$
for some fixed constant $P>1$.
By definition, $f$ decreases exponentially fast at infinity, like $e^{- \beta z}$
and is bounded by $\delta^{-1} / ( 1 + (z / \delta)^P)$ for small $z$.
 It is also possible to consider exponential weights. Note that $\beta$ may be arbitrarily small.

 We denote by ${\cal B}^{\beta,\gamma}$ the space that consists of functions  
 with finite $\| . \|_{\beta,\gamma}$ boundary layer norm.  
We expect the vorticity of the Navier-Stokes equations \eqref{linearizedNS} 
to be in the boundary layer space ${\cal B}^{\beta,\gamma}$, for each $x$ and $t$. 

Finally, we denote by $\cA^{\beta}$, the function space without a boundary layer behavior,
with the weighted norm 
$$
\| f \|_{\beta}  = \sup_{z\ge 0} | f(z) | e^{\beta z} .
$$

%%%%%

\subsection{Linearized Navier-Stokes}\label{sec-NSlin}

%%%%%

We shall work with the vorticity formulation of \eqref{linearizedNS}. 
Thanks to the divergence-free condition, we may introduce the stream function $\phi(t,x,z)$ and define 
$$
v = \nabla^\perp \phi = (\partial_z \phi, -\partial_x \phi).
$$ 
By definition, there holds 
\begin{equation}\label{ellipticstream}
\Delta \phi = \omega. \end{equation}
Taking the curl of \eqref{linearizedNS} yields 
\begin{equation}\label{eqs-vorticity} 
\partial_t \omega - L \omega = 0, \qquad L \omega : = - U \partial_x \omega + U'' \partial_x \phi +  \nu \Delta \omega,
\end{equation}
together with the boundary conditions 
\begin{equation}\label{BC-phi} \partial_x\phi_{\vert_{z=0}} = \partial_z \phi_{\vert_{z=0}} = 0. \end{equation} 
The problem \eqref{ellipticstream}-\eqref{BC-phi} is equivalent to the linearized Navier-Stokes problem \eqref{linearizedNS}-\eqref{BC} 
around $U(z)$. Similarly, the linearized Euler equations around $U(z)$ are \eqref{ellipticstream}-\eqref{eqs-vorticity}, with $\nu =0$ 
and boundary condition $\partial_x\phi_{\vert_{z=0}} =0$.

It is then convenient to introduce the Fourier transform in the $x$ variable.
Solutions to the linearized problem will be constructed in terms of Fourier series
\beq \label{Fourier-w}
\omega(t,x,z) = \sum_{\alpha\in \ZZ} e^{i\alpha x} \hat \omega_\alpha(t,z)
\eeq
where the Fourier coefficients $\hat \omega_\alpha(t,z)$ solve 
\begin{equation}\label{eqs-vorticitya} 
\partial_t \hat \omega_\alpha - L_\alpha \hat \omega_\alpha = 0, \qquad L_\alpha \hat \omega_\alpha 
: = - i\alpha U\hat \omega_\alpha + i\alpha \hat \phi_\alpha U'' +  \nu \Delta_\alpha \hat \omega_\alpha
\end{equation}
with 
$$
\Delta_\alpha \hat \phi_\alpha = \hat \omega_\alpha,
$$
together with the boundary conditions
$$
\alpha \hat \phi_\alpha = \partial_z\hat  \phi_\alpha =0
$$
at $z=0$. Here, 
$$
\Delta_\alpha = \partial_z^2 - \alpha^2.
$$ 
Observe that at $\alpha=0$, the linear problem \eqref{eqs-vorticitya} becomes 
$$ 
\partial_t \hat \omega_0 -  \nu \partial_z^2 \hat \omega_0 = 0
$$
whose semigroup can be explicitly solved. In particular, $\hat \omega_0(t,z) =0$ for all positive times, if it is initially zero. 

The aim of this paper is to prove the following result

\begin{theorem} \label{theo-main} 
Let $U(z)$ be a $C^\infty$ smooth boundary layer profile such that 
$U(0)=0$ and 
\begin{equation}\label{def-Ubl}
  |\partial_z^k (U(z) - U_+)| \le C_k e^{-\eta_0 z} , \qquad \forall~ z\ge 0, \quad k\ge 0,
  \end{equation}
for some constants $C_k, U_+, \eta_0$.
 Let $\lambda_0$ be the maximal unstable eigenvalue of the linearized Euler equations around $U$, 
namely the eigenvalue which has the largest real part $\Re\lambda_0$. 
We assume that $U$ is spectrally unstable for Euler equations, namely that 
$$
\Re \lambda_0 > 0.
$$
 Let $\alpha$ be fixed and let $\tau > 0$. Then there is a constant $C_{\tau}$,  so that  for any 
 $\nu \le 1$,  
\begin{equation}\label{main-bound}
\| e^{L_\alpha t} \omega_\alpha(0,.) \|_{\beta,\gamma}
 \le C_\tau e^{(\Re \lambda_0 + \tau) t }  \|\omega_\alpha(0,.) \|_{\beta,\gamma} ,  
\end{equation}
for any initial vorticity $\omega_\alpha(0,.)$  and for all $t\ge 0$, provided $\beta$ and $\gamma$ are small enough.  
\end{theorem}

Note that estimate (\ref{main-bound}) is uniform in $\nu$ as $\nu$ goes to $0$.
Note also that the vorticity is unbounded near the boundary as $\nu$ goes to $0$. 
 Theorem \ref{theo-main} provides a semigroup estimate for the linearized Navier-Stokes problem near an unstable 
 boundary layer profile, which is uniform in the vanishing viscosity limit. 
 Such an estimate is sharp without the knowledge of the multiplicity of the maximal unstable eigenvalue $\lambda_0$. 
 The interest in deriving such a sharp bound on the linearized Navier-Stokes problem is pointed out in 
 \cite{Grenier00CPAM,GGN1,GGN3,GrN1}. Certainly, the estimate \eqref{main-bound} is very natural, but, 
up to the best of our knowledge, it has never been proven in the literature . 
The difficulty lies in the fact that the initial data $\omega$  
has a boundary layer type behavior and 
hence in order to propagate this boundary layer behavior, pointwise bounds on the Green 
function of linearized Navier-Stokes equations near a boundary layer are needed.

%%%%%%%%%%

\subsection{The resolvent}

%%%%%%%%%%

In order to study $e^{L_\alpha t}$, it is convenient to 
take the Laplace transform of \eqref{eqs-vorticitya}. This leads to  the resolvent equation 
\begin{equation}\label{resolvent} 
(\lambda-L_\alpha)  \omega_\alpha =  f_\alpha
\end{equation} 
with $f_\alpha =  \omega_\alpha(0,z)$. As $L_\alpha$ is a compact perturbation of the Laplacian $\Delta_\alpha$, 
standard energy estimates yield that the operator $(\lambda - L_\alpha)^{-1}$ 
is well-defined and bounded from $H^{-1}$ to $H^{-1}$ by $|\Re \lambda - \gamma_0|^{-1}$,
 for some possibly large constant $\gamma_0$ and for any $\Re \lambda > \gamma_0$. 
 Hence, the classical semigroup theory (see, for instance, \cite[Theorem 6.13]{Pazy} or \cite{Z2}) yields   
\beq \label{int11}
e^{L_\alpha t}  f_\alpha = {1 \over 2  \pi i} \int_{\Gamma_\alpha} e^{\lambda t}  (\lambda - L_\alpha)^{-1}  f_\alpha \, d \lambda
\eeq
where $\Gamma_\alpha$ is a contour lying on the right of the spectrum of $L_\alpha$. 
It is traditional to introduce 
\begin{equation}\label{def-c}
c =  i \alpha^{-1}Ê\lambda , \qquad \eps = {{\nu} \over i \alpha} .
\end{equation}
Writing $\omega_\alpha = \Delta_\alpha  \phi_\alpha$, the resolvent equation \eqref{resolvent} 
becomes the classical Orr-Sommerfeld equations for the stream function $\phi_\alpha$
%\beq \label{OSphys}
%%\OS(\phi_\alpha) :=
% - \eps \Delta_\alpha^2 \phi_\alpha 
%+ (U - c) \Delta_\alpha \phi_\alpha - U'' \phi_\alpha = { f_\alpha \over i \alpha},
%\eeq
%In this paper we multiply (\ref{OSphys}) by $\alpha$ 
%and get
\beq \label{OS}
i \nu \Delta_\alpha^2 \phi_\alpha 
+ (\alpha U - i \lambda) \Delta_\alpha \phi_\alpha - \alpha U'' \phi_\alpha = - i f_\alpha,
\eeq
together with the boundary conditions: 
\begin{equation}\label{OS2}
 \phi_\alpha = \partial_z  \phi_\alpha =0, \qquad \mbox{on} \quad z=0.
\end{equation}
% We underline the fact that $\eps$ is a purely imaginary number.
% In the sequel, we denote by $\sqrt{i \nu}$ the square root of
%$i \nu$ which has a positive real part, namely
%$$
%\sqrt{i \nu} = {1 \over \sqrt{2}} (1 - i) \sqrt{  \nu } .
%$$
We then solve the Orr-Sommerfeld equations thanks to their Green functions. 
For each fixed $\alpha \in \NN$ and $c\in \CC$, we let $G_{\alpha,\lambda}(x,z)$ 
be the corresponding Green kernel of the Orr-Sommerfeld problem \eqref{OS}-\eqref{OS2}. 
By definition, for each $x\in \RR_+$, $G_{\alpha,c}(x,z)$ solves 
$$ 
\OS(G_{\alpha,\lambda} (x,\cdot)) = \delta_x (\cdot)
$$
on $z\ge 0$, together with the boundary conditions:
$$
G_{\alpha,\lambda}(x,0) = \partial_z G_{\alpha,\lambda} (x,0) =0, \qquad \lim_{z\to \infty} G_{\alpha,\lambda}(x,z) =0.
$$
The solution $\phi_\alpha$ to the Orr-Sommerfeld problem \eqref{OS}-\eqref{OS2} is then constructed by 
$$ 
\phi_\alpha(z) = -i  \int_0^\infty G_{\alpha, c} (x,z)  f_\alpha(x) \; dx .
$$
Similarly
\begin{equation}\label{int-La} 
 (\lambda - L_\alpha)^{-1}  f_\alpha (z) = - i   \int_0^\infty \Delta_\alpha G_{\alpha,\lambda} (x,z)  f_\alpha (x)\;dx ,
 \end{equation}
in which $\Gamma_\alpha$ is chosen depending on $\alpha$ and lying in the resolvent set of $L_\alpha$. 

Such a spectral formulation of the linearized Navier-Stokes equations near a boundary layer shear profile has been 
intensively studied  in the physical literature.  We in particular refer to
\cite{Reid, Sch, Lin0,LinBook} for the major works of Heisenberg, Tollmien, C.C. Lin, and Schlichting on the subject. 
We also refer to \cite{GGN1, GGN3, GGN2} for the rigorous spectral analysis on the Orr-Sommerfeld equations.

We now recall the main results from \cite{GrN1} on the Green function $G_{\alpha,\lambda}$. 
We focus on the case $\alpha > 0$, the case $\alpha < 0$ being similar.
We introduce Rayleigh's equation 
$$
Ray_\alpha(\phi) = (\alpha U- i \lambda )\Delta_\alpha\phi - \alpha U''\phi =0.
$$
The Rayleigh equation $Ray_\alpha(\phi) = 0$ has two solutions $\phi_{\alpha,\pm}$, with respective
behaviors $e^{\pm \alpha z}$ at infinity. 
We define the Evans function $E(\alpha,\lambda)$ by
$$
E(\alpha,\lambda) = \phi_{\alpha,-}(0) .
$$
In this paper, we restrict ourselves to the case when $\lambda$ is away from the range of $-i\alpha U$. 
Precisely, let $\epsilon_0$ be an arbitrarily small, but fixed, positive constant, 
we shall consider the range of $(\alpha,\lambda)$ in $\RR_+\times \CC$ so that 
\begin{equation}\label{noncritical}
d(\alpha,\lambda) =\inf_{z\in \RR_+}| \lambda + i \alpha U(z)| \ge \epsilon_0.
\end{equation}
Note that $d(\alpha,\lambda) = \Re \lambda$ if $\Im \lambda \in - \alpha \mathrm{Range}(U)$. In any case, we have 
\beq \label{range-c1}
d(\alpha,\lambda) \ge | \Re \lambda | .
\eeq
The main result from \cite{GrN1} is as follows.

\begin{theorem}\label{theo-GreenOS}
Let $U(z)$ be a %non degenerate
boundary layer profile which satisfies \eqref{def-Ubl}. 
For each $\alpha,\lambda$, let by $G_{\alpha,\lambda}(x,z)$ be
the Green kernel of the Orr-Sommerfeld equation, with source term in $x$, and let 
\begin{equation}\label{def-mMf}
\mu_s = | \alpha| , \qquad   \mu_f(z) =  \nu^{-1/2} \sqrt{\lambda + \nu \alpha^2 + i \alpha U(z)}, 
\end{equation}
where we take the square root with positive real part. Let $0 < \theta_0 < 1$ and $\zeta < 1/2$. 
Let $\sigma_0 > 0$ be arbitrarily small.
Then, there exists $C_0 > 0$ so that 
\begin{equation}\label{est-GrOS}
 |G_{\alpha,\lambda}(x,z)| 
  \le \frac{C_0}{  \mu_s  d(\alpha,\lambda) }e^{-\theta_0 \mu_s |x-z|} 
  +    \frac{C_0}{ |\mu_f(x) |  d(\alpha,\lambda)}  e^{- \theta_0 | \int_x^z \Re \mu_f \; dy|} 
\eeq
uniformly for all $x,z\ge 0$ and $0 < \nu \le 1$, and uniformly in $(\alpha,\lambda)\in \RR\setminus\{0\}\times \CC$ so that $| \alpha | \le\nu^{-\zeta}$, \eqref{noncritical} holds, and
$$
| E(\alpha,\lambda)| > \sigma_0 .
$$
%In particular, we have
%\begin{equation}\label{est-GrOSsimply}
% |G_{\alpha,\lambda}(x,z)| 
%  \le \frac{C_0}{  \mu_s  | \Re \lambda |}e^{-\theta_0 \mu_s |x-z|} 
%  +    \frac{C_0}{ |\mu_f(x) |  | \Re \lambda |}  e^{- \theta_0 | \int_x^z \Re \mu_f \; dy|} .
%\eeq
In addition, there hold the following derivative bounds
\begin{equation}\label{est-GrOS-d}
 | \partial_x^k \partial_z^\ell G_{\alpha,\lambda}(x,z)|  
 \le \frac{C_0\mu_s^{k+\ell}}{   \mu_s  d(\alpha,\lambda)}  e^{-\theta_0 \mu_s |x-z|} 
 +    \frac{C_0 | \mu_f(z) |^{k+\ell}}{| \mu_f(x)| d(\alpha,\lambda) }  e^{- \theta_0 | \int_x^z \Re \mu_f \; dy|}  
\eeq
for all $x,z\ge 0$ and $k, \ell \ge 0$, in which $M_f = \sup_z \Re  \mu_f(z) $. 
In particular, we have 
\begin{equation}\label{est-GrOS-delta}
{ | \Delta_\alpha G_{\alpha,\lambda}(x,z)|  
 \le \frac{C_0}{   d(\alpha,\lambda)^2}  e^{-\theta_0 \mu_s |x-z|} }
 +    \frac{C_0 }{ \nu | \mu_f(x)| }   e^{- \theta_0 | \int_x^z \Re \mu_f \; dy|}  
\eeq
where we "gain" a factor $\mu_s$ in the first term on the right hand side.
\end{theorem}

%
%Let us comment the prefactors of (\ref{est-GrOS}). First we may divide (\ref{OS}) by $(\alpha U - i \lambda)$ which never vanishes.
%This leads to the $d(\alpha,\lambda)$ factor. Then on the slow modes we have mainly to invert $\partial^2 - \alpha^2$ which
%leads to a gain of $\mu_s  = | \alpha |$. The inversion of the fast modes leads to a gain of  $\mu_f$.
%Moreveover, we gain an $\alpha$ factor on the slow mode when we take $\partial^2 - \alpha^2$.

%%%%%%%%%%%%%%%%%%%%%%%

\section{Semigroup bounds}\label{sec-semigroup}

%%%%%%%%%%%%%%%%%%%%%%%

In this section, we shall bound the semigroup $e^{L_\alpha t}$ for $| \alpha | \le\nu^{-\zeta}$ with $\zeta<1/2$. 
In view of (\ref{est-GrOS-delta}), we decompose $e^{L_\alpha t}$ as follows: 
\begin{equation}\label{de-eLta}
e^{L_\alpha t}  = S_{\alpha,1} + S_{\alpha,2} 
\end{equation} 
with 
\begin{equation}\label{def-S1212}
\begin{aligned}
S_{\alpha,1}  \omega_\alpha (z): &=\frac{1}{2\pi i}  \int_{\Gamma_{\alpha}} \int_0^\infty e^{\lambda t}  
 \mathcal{S}_1(x,z) \omega_\alpha (x) \; dxd \lambda,
\\S_{\alpha,2}  \omega_\alpha (z): &= \frac{1}{2\pi i}  \int_{\Gamma_{\alpha}} \int_0^\infty e^{\lambda t}   
\mathcal{S}_2(x,z) \omega_\alpha (x) \; dxd \lambda,
\end{aligned}
\end{equation}
where the kernels $ \mathcal{S}_j(x,z)$ are meromorphic in $\lambda$ and satisfy 
\begin{equation}\label{est-decompGrOS}
\begin{aligned}
|\mathcal{S}_1(x,z) |
&\le  \frac{C_0}{   d(\alpha,\lambda)^2 } e^{-\theta_0 \mu_s |x-z|}, \\
|\mathcal{S}_2(x,z)| &\le \frac{C_0}{\nu | \mu_f(x)| }   
e^{- \theta_0 | \int_x^z \Re \mu_f \; dy|} .
\end{aligned}\eeq

%%%%%%%%%%%

\subsection{Bounds on $S_{\alpha,1}$.}

%%%%%%%%%%%

In this section, we prove the following. 

\begin{proposition}\label{prop-S1a} 
Let $\beta \in (0,\frac12]$. % and let $\lambda_\alpha$ be the maximal unstable eigenvalue of $L_\alpha$. 
For any positive $\tau$, there is a constant $C_\tau$ so that 
\begin{equation}\label{est-S1a}
\| S_{\alpha,1}\omega_\alpha \|_{\beta,\gamma}
\le C_\tau e^{(\Re \lambda_0 +\tau)t} \| \omega_\alpha \|_{\beta,\gamma},
 \end{equation}
uniformly in $t\ge 0$, small $\nu >0$, and $|\alpha|\le \nu^{-\zeta}$ with $\zeta<1/2$. 
\end{proposition}
\begin{proof} 
We restrict ourselves to $\alpha > 0$.
Since the Green kernel $\mathcal{S}_1(x,z) $ is meromorphic in $\lambda$, we can apply the Cauchy's theory and take the contour $\Gamma_\alpha$ of integration to consist of $\lambda$ so that 
\begin{equation}\label{def-tau0}
\Re \lambda = \Re \lambda_0 +\tau
\end{equation}
for arbitrary small, but fixed, constant $\tau>0$. Since $\Gamma_\alpha$ remains in the resolvent set of $L_\alpha$, 
the (inviscid) Evans function $E(\alpha,\lambda)$ never vanishes. In addition, recalling the assumption \eqref{noncritical} and writing $\lambda = \Re \lambda + i \Im \lambda$, we have
$$
d(\alpha,\lambda) =\inf_{z\in \RR_+}| \lambda + i \alpha U(z)| \ge \theta_0 (1 + \inf_z |\Im \lambda - \alpha U(z)|). $$ 
We thus obtain from \eqref{est-decompGrOS} that 
\begin{equation}\label{bd-S1X1}
|\mathcal{S}_1(x,z) |\le C_\tau (1 +  \inf_z |\Im \lambda - \alpha U(z)|)^{-2} e^{-\mu_s |x-z|},
\end{equation}
for all $\lambda\in \Gamma_\alpha$.

Let us now estimate the convolution $S_{\alpha,1}  \omega_\alpha (z)$ in the boundary layer norm 
$\| \cdot \|_{\beta,\gamma}$. First we recall from the definition that 
$$
| \omega_\alpha(x) | \le \|\omega_\alpha\|_{\beta,\gamma} e^{-\beta x} (1+ \delta^{-1}\phi_{P}(\delta^{-1}x)).
$$ 
Hence, recalling \eqref{def-S1212} and using \eqref{bd-S1X1}, we have 
$$\begin{aligned}
|S_{\alpha,1}  \omega_\alpha (z)| 
&\le C_\tau\| \omega\|_{\beta,\gamma} \int_{\RR} \int_0^\infty  
(1 +  \inf_z |\Im \lambda - \alpha U(z)|)^{-2}  e^{(\Re \lambda_0 + \tau)t} 
\\&\quad \times e^{-\mu_s |x-z|}e^{-\beta |x|} (1+ \delta^{-1}\phi_{P}(\delta^{-1}x))\; dx \, \, d \Im \lambda 
.\end{aligned}$$
The integral in $\Im\lambda$ is  bounded, yielding 
$$\begin{aligned}
|S_{\alpha,1}  \omega_\alpha (z)| 
&\le C_\tau\| \omega\|_{\beta,\gamma} e^{(\Re \lambda_0 + \tau)t}  \int_0^\infty  
 e^{-\mu_s |x-z|}e^{-\beta |x|} (1+ \delta^{-1}\phi_{P}(\delta^{-1}x))\; dx
.\end{aligned}$$
Recall that $\mu_s = |\alpha|$. Using the triangle inequality $|z|\le |x|+ |x-z|$ 
and the fact that $\beta\le1/2 < | \alpha |$, we obtain 
$$ 
e^{-\mu_s |x-z|} 
e^{-\beta |x|} \le e^{-\beta z} e^{-\frac12\alpha |x-z|}.
$$
Hence, we get 
$$\begin{aligned}
|S_{\alpha,1}  \omega_\alpha (z)| 
&\le C_\tau\| \omega\|_{\beta,\gamma}e^{(\Re \lambda_0 + \tau)t} 
 e^{-\beta z} \int_0^\infty  e^{-\frac12 \alpha |x-z|}(1+ \delta^{-1}\phi_{P}(\delta^{-1}x))\; dx
 \\
 &\le C_\tau\| \omega\|_{\beta,\gamma}e^{(\Re \lambda_0 + \tau)t} 
 e^{-\beta z} \Big( \alpha^{-1} + \delta^{-1}\int_0^\infty \phi_{P}(\delta^{-1}x)\; dx\Big).
\\
&\le C_\tau\| \omega\|_{\beta,\gamma}e^{(\Re \lambda_0 + \tau)t}  e^{-\beta z} ,
\end{aligned}$$
completing the proof of the Proposition. 
\end{proof}

%%%%%%%%%%%

\subsection{Bounds on $\mathcal{S}_{\alpha,2} $.}

%%%%%%%%%%%

In this section, we give bounds on the semigroup $\mathcal{S}_{\alpha,2}$, defined as in \eqref{def-S1212}. Precisely, we have 

\begin{proposition}\label{prop-S2a} Let $\beta \in (0,\frac12]$. 
For any positive $\tau$, there is a constant $C_\tau$ so that 
\begin{equation}\label{est-S2a}
\| S_{\alpha,2}\omega_\alpha \|_{\beta,\gamma}
\le C_\tau e^{(\Re \lambda_0 +\tau)t} \| \omega_\alpha \|_{\beta,\gamma},
 \end{equation}
uniformly in $t\ge 0$, small $\nu >0$, and $|\alpha|\le \nu^{-\zeta}$ with $\zeta<1/2$.
\end{proposition}
To prove this proposition we will use the following Lemma

\begin{lemma}\label{lem-tGreen} 
Let $\mathcal{S}_{2}(x,z)$ be the Green kernel defined as in \eqref{def-S1212}. Introduce the temporal Green function 
\begin{equation}\label{def-tGreen2}G_2(t,x,z) :=  \frac{1}{2\pi i}  \int_{\Gamma_{\alpha}} e^{\lambda t}   
\mathcal{S}_2(x,z) \; d \lambda  .\end{equation}
Then, for any positive $\tau$, there are constants $C_\tau, \theta_\tau$ so that there holds 
\begin{equation}\label{ptw-G22}
\begin{aligned}
& |G_2(t,x,z) |
\\&\le C_\tau (\nu t)^{-1/2}e^{\tau t} e^{-\frac{|x-z|^2}{16 \nu t}} 
+   C_\tau \sum_{\Re \lambda_\alpha\ge \tau}e^{(\Re \lambda_\alpha + \tau)t}\nu^{-1/2} e^{-\theta_{\tau}\nu^{-1/2}|x-z|} 
 \end{aligned}\end{equation}
uniformly in $t\ge 0$, small $\nu >0$, and $\alpha \in \ZZ^*$, in which 
$$
\theta_{\tau} = \frac12 \sqrt{\Re\lambda_\alpha + \tau + \alpha^2 \nu},
$$
and the summation is taken over finitely many unstable eigenvalues $\lambda_\alpha$ of $L_\alpha$ 
such that $\Re \lambda_\alpha \ge \tau$. 
\end{lemma}

\begin{proof} 
We move the contour of integration $\Gamma_\alpha$ in \eqref{def-tGreen2} 
from its initial position to the particular one defined below at (\ref{def-aaa}).
As $\Gamma_\alpha$ moves to the left in the complex plane, it may meet unstable eigenvalues of $L_\alpha$

We first bound the contribution of these unstable eigenvalues. 
To proceed, let $\tau$ be an arbitrary positive number, and let $\lambda_\alpha = -i\alpha c$ be a zero of $E(\alpha,\lambda)$
 such that $\Re \lambda_\alpha \ge \tau$. 
 As $E(\alpha,\lambda)$ is analytic in $c$, its zeros $\lambda_\alpha$ are isolated. 
 Taking $\tau$ smaller if needed, we can assume that there is no other unstable eigenvalue
 in the ball $B(\lambda_\alpha, \frac12\tau)  = \{|\lambda - \lambda_\alpha |\le \frac12\tau\}$. 
 In particular, we have $|E(\alpha,\lambda)| \ge C_\tau$
for $\lambda \in \partial B(\lambda_\alpha, \frac12\tau) $. 
In addition, since $\Re \lambda \ge \frac14(\Re \lambda_\alpha + \tau)$ on $ \partial B(\lambda_\alpha, \frac12\tau) $, we have 
$$ 
\Re \mu_f = \nu^{-1/2} \Re \sqrt{\lambda + \alpha^2\nu+ i\alpha U } \ge \nu^{-1/2}\theta_\tau,
$$
with $\theta_\tau = \frac12\sqrt{\Re\lambda_\alpha + \tau + \alpha^2 \nu}.$  
Thus, on $\partial B(\lambda_\alpha, \frac12\tau)$, there holds
$$
|\mathcal{S}_2(x,z)| \le \frac{C_0 }{ \nu | \mu_f(x)| }   e^{- \theta_0 | \int_x^z \Re \mu_f \; dy|}  
\le C_\tau  \nu^{-1/2} e^{- \theta_\tau  \nu^{-1/2} |x-z| } ,
$$
upon recalling \eqref{noncritical}. This yields
\begin{equation}\label{bd-GreenRes}
\Big| \int_{\partial B(\lambda_\alpha, \frac12\tau)} e^{\lambda t}   
\mathcal{S}_2(x,z) \; d \lambda \Big|\le C_\tau \nu^{-1/2}  e^{(\Re \lambda_\alpha + \tau)t}
   e^{- \theta_\tau \nu^{-1/2} |x-z| }  , 
\end{equation}  
which contributes to the last term in the Green function bound \eqref{ptw-G22}.  

We are now ready to choose a suitable contour of integration $\Gamma_\alpha$. 
%Since the second term in the Green kernel is clearly better behaved, let us focus on treating the first term, namely on
%$$
%G_{21}(t,x,z) =  \frac{1}{2\pi i}  \int_{\Gamma_{\alpha}} e^{\lambda t}   
%\mathcal{S}_{21}(x,z) \; \frac{d \lambda}{i\alpha } 
%$$
%with 
%\begin{equation}\label{bd-kernelS21}
%| \mathcal{S}_{21}(x,z) |  \le  { C \over |\epsilon \mu_f(x) D(\alpha,\lambda)|}  e^{-|\int_x^z \Re \mu_f(y) \; dy|} 
%.\end{equation}
Let us consider the case when $x<z$, the other case is similar.
Recall that 
$$
\mu_f(z) %= \sqrt{\alpha^2 + (U-c)/\epsilon} 
= \nu^{-1/2}\sqrt{\lambda + i\alpha U+ \alpha^2  \nu } .
$$ 
By construction, $\mathcal{S}_{2}(x,z)$ is holomorphic in $\lambda$, except on the complex half strip
$$ 
\mathcal{H}_\alpha(x,z): = \Big\{ \lambda = -k -\alpha^2  \nu + i \alpha U(y), \qquad k\in \RR_+, \quad y \in [x,z]\Big\}.
$$
In our choice of contour of integration below, we shall avoid to enter this complex strip. 
We set 
$$
\begin{aligned}
\Gamma_{\alpha,1} &:= \Big\{ \lambda =\gamma_1 -\alpha^2 \nu - i \alpha c,
 \qquad \min_{y\in [x,z]} U(y) \le c \le \max_{y\in [x,z]} U(y) \Big\} 
\\
\Gamma_{\alpha,2} &:=  \Big\{ \lambda =\gamma_1 -\alpha^2 \nu - k^2  \nu  - i \alpha \min_{[x,z]} U + 2  \nu i ak,
 \qquad k\ge 0 \Big\}
\\
\Gamma_{\alpha,3} &:= \Big\{ \lambda = \gamma_1 -\alpha^2 \nu - k^2  \nu  - i \alpha \max_{[x,z]} U + 2  \nu i ak, 
\qquad k\le 0 \Big\} 
\end{aligned}$$
where
\begin{equation}\label{def-aaa} 
\gamma_1 = \tau + a^2\nu + \alpha^2\nu, \qquad a = \frac{|x-z| + \sqrt{\nu t}}{4 \nu t} .
\end{equation}
%and where $\gamma_1$ will be chosen later 
%(see Figure \ref{fig-contour}). 
The choice of the parabolic contours $\Gamma_{\alpha,2}$ and $\Gamma_{\alpha,3}$ 
is necessary to avoid singularities in small time \cite{ZH}. 
Note that they never meet the complex strip $\mathcal{H}_\alpha (x,z)$.
In addition, they may leave unstable eigenvalues to the right, in which case the contribution from unstable eigenvalues \eqref{bd-GreenRes} 
is added into the bounds on the Green function.

%
%\begin{figure}[t]
%\centering
%\includegraphics[scale=.5]{contourOS}
%\put(-20,1){$\Gamma_\alpha$}
%\put(-15,117){$\CC$}
%\put(-125,0){$\Gamma_{\alpha,3}$}
%\put(-125,125){$\Gamma_{\alpha,2}$}
%\put(-65,45){$\Gamma_{\alpha,1}$}
%\put(-105,50){$0$}
%\put(-150,75){$\mathcal{H}_\alpha$}
%\put(-85,100){$\Gamma_{\alpha,4}$}
%\put(-85,25){$\Gamma_{\alpha,5}$}
%\caption{\em Shown the decomposition of contour $\Gamma_\alpha$ of integration.}
%\label{fig-contour}
%\end{figure}
%

%
%\begin{itemize}
%
%\item Case 1: $a^2  \nu \ge \theta_0$. In this case, we take 
%\begin{equation}\label{def-ga1} 
%\gamma_1 : =  a^2  \nu  +\frac12 \alpha^2  \nu  
%\end{equation}
%and the contour of integration is taken to be 
%$$ 
%\Gamma_\alpha = \Gamma_{\alpha,1} \cup \Gamma_{\alpha, 2} \cup \Gamma_{\alpha,3}.
%$$
%
%
%
%\item Case 2: $a^2  \nu \le \theta_0$. In this case, 
%\begin{equation}\label{def-ga2} 
%\gamma_1 : =  \theta_0  +\frac12 \alpha^2  \nu  
%\end{equation}
%and the contour of integration is taken to be 
%$$ 
%\Gamma_\alpha = \Gamma_{\alpha,1} \cup \Gamma_{\alpha, 2} \cup \Gamma_{\alpha,3} \cup \Gamma_{\alpha,4} \cup \Gamma_{\alpha,5}
%$$ 
%in which 
%$$
%\begin{aligned}
%\Gamma_{\alpha,4} &:= \Big\{ \lambda = k  - \frac12 \alpha^2  \nu  - i \alpha \min_{[x,z]} U, \qquad a^2  \nu \le k \le \theta_0\Big\}
%\\
%\Gamma_{\alpha,5} &:= \Big\{ \lambda = k - \frac12\alpha^2  \nu  - i \alpha \max_{[x,z]} U , \qquad a^2  \nu \le k \le \theta_0 \Big\} .
%\end{aligned}$$
%
%\end{itemize}
%We stress that in both cases, we have 
%\begin{equation}\label{def-ga} 
%\gamma_1 \ge  \theta_0  +\frac12 \alpha^2  \nu  .
%\end{equation}
%

%%%%%%

\subsubsection*{Bounds on $\Gamma_{\alpha,1}$.}

%%%%%%

We start our computation with the integral on $\Gamma_{\alpha,1}$. 
We first note that for $\lambda \in \Gamma_{\alpha,1}$, there holds
$$  
\Re  \mu_f = \nu^{-1/2} \Re   \sqrt{ \gamma_1 + i\alpha (U-c)}  \ge \nu^{-1/2} \sqrt{\gamma_1} .
$$
%in which $c$ ranges within the range of $U(y)$ with $y\in [x,z]$.  
%On the other hand, %using $\epsilon = \nu/i\alpha$, 
%$\Re \mu_f(x)  \ge \nu^{-1/2} \sqrt {\gamma_1}$
%and $| \mu_f(z) | \le C \nu^{-1/2}$, 
Hence, we have 
$$
\begin{aligned}
|\mathcal{S}_2(x,z)| 
&\le \frac{C_0 }{ \nu | \mu_f(x)|  }   
e^{- \theta_0 | \int_x^z \Re \mu_f \; dy|} 
 \le C  \nu^{-1/2} \gamma_1^{-1/2}  e^{-\nu^{-1/2} \sqrt{\gamma_1} |x-z|} .
\end{aligned}$$
Using $\gamma_1 \ge a^2\nu$ and $\gamma_1 \ge \alpha^2 \nu$, we note that 
$$
\begin{aligned}
 e^{-\frac12\nu^{-1/2} \sqrt{\gamma_1} |x-z|} &\le e^{-\frac a2 |x-z|} = e^{-\frac{|x-z|^2}{8\nu t} - \frac{|x-z|}{8\sqrt{\nu t}}}
 \\
 e^{-\frac12\nu^{-1/2} \sqrt{\gamma_1} |x-z|} &\le e^{-\frac12 \alpha|x-z|}  .
  \end{aligned}$$
On the other hand, recalling \eqref{def-aaa}, we compute 
$$\begin{aligned}
| e^{\lambda t} |  = e^{\gamma_1 t} e^{-\alpha^2 \nu t}  = e^{\tau t} e^{a^2 \nu t} = e^{\tau t} e^{\frac{|x-z|^2}{16\nu t} + \frac{|x-z|}{8\sqrt{\nu t}} + \frac1{16}}.
\end{aligned}
  $$
 Thus, putting the above estimates together, we obtain 
\begin{equation}
\begin{aligned}
|  e^{\lambda t} \mathcal{S}_{2}(x,z) |  &\le C  \nu^{-1/2} \gamma_1^{-1/2} e^{\tau t}  e^{-\frac{|x-z|^2}{16\nu t} - \frac12 \alpha|x-z|}
\end{aligned}\end{equation}
for any $\lambda \in \Gamma_{\alpha,1}$. 
Hence, we estimate 
$$\begin{aligned}
 \Big| \int_{\Gamma_{\alpha,1}} e^{\lambda t}   \mathcal{S}_{2}(x,z) \; d\lambda \Big|
&\le  C \alpha \nu^{-1/2} \gamma_1^{-1/2}e^{\tau t} e^{-\frac{|x-z|^2}{16\nu t} - \frac12 \alpha|x-z|}
\int_{\min_{[x,z]} U}^{\max_{[x,z]}U} dc 
 \\&\le C \alpha \nu^{-1/2} \gamma_1^{-1/2}e^{\tau t}e^{-\frac{|x-z|^2}{16\nu t} - \frac12 \alpha|x-z|} |x-z| \| U'\|_{L^\infty}
   .      \end{aligned}
 $$ 
 Using the inequality $X e^{-X} \le C $ for $X\ge 0$, we have 
 $$ 
 e^{-\frac12 \alpha  |x-z|} \alpha |x-z| \le C.
 $$
We thus obtain 
\begin{equation}\label{bd-Ga1}\begin{aligned}
 \Big| \int_{\Gamma_{\alpha,1}} e^{\lambda t}   \mathcal{S}_{2}(x,z) \; d\lambda\Big|
&\le   C \nu^{-1/2} \gamma_1^{-1/2} e^{\tau t}e^{-\frac{|x-z|^2}{16\nu t}}
.\end{aligned}
  \end{equation}
This yields the claimed estimate on $\Gamma_{\alpha,1}$, upon noting that $\sqrt t \le C_\tau e^{\tau t}$ for any $t\ge 0$.

%%%%%%%%%%%%
  
\subsubsection*{Bounds on $\Gamma_{\alpha,2}$ and $\Gamma_{\alpha,3}$.}

%%%%%%%%%%%%

By symmetry, it suffices to give bounds on $\Gamma_{\alpha,2}$. For $\lambda \in \Gamma_{\alpha,2}$ and $y \in [x,z]$, we compute 
$$
\begin{aligned}
 \mu_f(y)  &= \nu^{-1/2}  \sqrt{ \gamma_1  - k^2 \nu+ i \alpha (U -  \min_{[x,z]} U) + 2 i  \nu  ak }
 \\&  =\nu^{-1/2} \sqrt{ \tau + \alpha^2 \nu + i \alpha (U -  \min_{[x,z]} U) + (a+ik)^2 \nu } .
 \end{aligned}
 $$
 Recalling that $k, \alpha, \tau\ge 0$, the above yields
  \beq \label{minia}
  \Re \mu_f(y) \ge \Re \sqrt{(a+ik)^2} = a
  \eeq
for $y\in [x,z]$.  
So, using $a = (|x-z| + \sqrt{\nu t}) / 4\nu t$, we get 
$$ 
\begin{aligned}
a^2 \nu t - a|x-z| 
&= \frac{|x-z|^2 + 2 \sqrt{\nu t}|x-z| + \nu t }{16 \nu t} - \frac{|x-z|^2 + \sqrt{\nu t}|x-z|}{4\nu t}
\\
&= \frac1{16}- \frac{3|x-z|^2 + 2\sqrt{\nu t}|x-z|}{16\nu t}
\end{aligned} $$
and hence 
\begin{equation}\label{bd-G2exp}\begin{aligned} 
e^{  \Re \lambda t} e^{- \int_x^z  \Re \mu_f(y) \; dy} 
&\le e^{\tau t} e^{a^2  \nu t- \nu k^2 t } e^{-a|x-z|} 
\le C e^{\tau t}e^{- \nu k^2 t } e^{-\frac{3|x-z|^2}{16 \nu t}} .
\end{aligned}\end{equation}
Moreover 
$$
\begin{aligned}
 \mu_f(x)  =\nu^{-1/2} \sqrt{ \tau + \alpha^2 \nu 
 + a^2 \nu - k^2 \nu  + i \alpha (U -  \min_{[x,z]} U) + 2 i\nu ak} .
 \end{aligned}$$
This implies that 
$$
|\mu_f(x)|\ge \nu^{-1/2}\sqrt{\alpha (U-\min_{[x,z]}U) + 2 \nu ak} \ge \sqrt{2 ak},
$$ 
recalling $k\ge 0$. We also have $|\mu_f(x)|\ge \Re \mu_f(x) \ge a$ (see \ref{minia}). Hence, 
$$
 |\mu_f(x)| \ge \frac12 (a + \sqrt {ak}).
$$
Hence, recalling $i\alpha \varepsilon  =  \nu$ and noting $ d\lambda = 2\nu i (a+ ik) dk $ on $\Gamma_{\alpha,2}$,
we can estimate 
%$$
%XXX...what...about...\mu_f(z)...?...XXX
%$$
$$
\begin{aligned}
 \Big| \int_{\Gamma_{\alpha,2}} e^{\lambda t}   \mathcal{S}_{2}(x,z) \; d\lambda\Big|
&\le  C_\tau e^{\tau t}e^{-\frac{|x-z|^2}{4 \nu t}}\int_{\RR_+} e^{- \nu k^2 t} { \frac{|a+ik|dk}{ a + \sqrt {ak}}} 
\\&\le C_\tau e^{\tau t}e^{-\frac{|x-z|^2}{4 \nu t}}\int_{\RR_+} e^{- \nu k^2 t}{( 1 + a^{-1/2}\sqrt k )} dk 
\\&\le  C_\tau (\nu t)^{-1/2}e^{\tau t}e^{-\frac{|x-z|^2}{4 \nu t}} {(1 + a^{-1/2} (\nu t)^{-1/4} )}
\end{aligned}
$$
up to the contribution from unstable eigenvalues \eqref{bd-GreenRes}. Note that $a\ge (\nu t)^{-1/2}$ and hence $a^{-1/2}(\nu t)^{-1/4}\le 1$.  
The Lemma follows.
\end{proof}

\begin{lemma}\label{lem-convolutionG2} 
Let 
$$
H(t,x,z): =(\nu t)^{-1/2} e^{-\frac{|x-z|^2}{M\nu t}} 
$$ for some positive $M$. 
For any positive $\beta$, there is a constant $C_0$ so that 
$$
\Big \| \int_0^\infty H(t,x,\cdot) \omega_\alpha(x)\; dx \Big\|_{ \beta, \gamma} 
\le C_0 e^{M\beta^2 \nu t } \| \omega_\alpha\|_{ \beta, \gamma}.
$$
\end{lemma}

%$$
%XXX...XXX: {red}{ .....is....there....todo....here....?}
%$$

\begin{proof} Let $\omega_\alpha (z)$ be a boundary layer function that satisfies
\begin{equation}\label{assmp-wbl}
|\omega_\alpha (z) | \le
 \| \omega_\alpha\|_{ \beta, \gamma} \Bigl( 1 +  \delta^{-1} \phi_{P} (\delta^{-1} z)  \Bigr) e^{-\beta z} .
 \end{equation}
We first show that the convolution has the right exponential decay at infinity. Indeed, in the case when $|x-z|\ge M\beta\nu t$, 
using $|x| \ge |z| - |x - z|$,
we have
$$
e^{-\frac{|x-z|^2}{M\nu t}} e^{-\beta  |x|} \le e^{-\beta  |z|} e^{-|x-z| \Big( \frac{|x-z|}{M\nu t} - \beta \Big)} \le e^{-\beta  |z|}.
$$
Whereas, for $|x-z| \le M \beta \nu t$, we note that 
$$ 
e^{-\frac{|x-z|^2}{M\nu t}} e^{-\beta  |x|} \le 
e^{- M \beta^2 \nu t} e^{-\beta  |x|} \le e^{-\beta |x-z|} e^{-\beta  |x|} \le e^{-\beta  |z|}.
$$
Combining, we have 
\begin{equation}\label{exp-beta13}
e^{-\frac{|x-z|^2}{M\nu t}} e^{-\beta x} \le 
%e^{M \beta^2 \nu t} 
e^{-\beta |z|}, \qquad \forall x,z\in \RR
\end{equation}
which yields the spatial decay $e^{-\beta z}$ in the norm $\|\cdot \|_{\beta,\gamma}$.

It remains to study the integral 
\begin{equation}\label{conv-heatbl1-stable}
\begin{aligned}
\int_0^\infty (\nu t)^{-1/2}  e^{-\frac{|x-z|^2}{M\nu t}} 
  \Bigl( 1 + \delta^{-1} \phi_{P} (\delta^{-1} x)  \Bigr) \; dx.
\end{aligned}\end{equation}
The integral without the boundary layer behavior is clearly bounded. Next, using the fact that $\phi_{P}(\delta^{-1}x)$ is decreasing in $x$, we have 
$$
\begin{aligned}
 \int_{z/2}^\infty & (\nu t)^{-1/2}   e^{-\frac{|x-z|^2}{M\nu t}} 
  \delta^{-1} \phi_{P} (\delta^{-1} x) \; dx 
 \\ &\le C_0   \delta^{-1} \phi_{P} (\delta^{-1} z)  \int_{z/2}^\infty  (\nu t)^{-1/2}   e^{-\frac{|x-z|^2}{M\nu t}}   \; dx
  \\&\le C_0   \delta^{-1} \phi_{P} (\delta^{-1} z) .  \end{aligned}$$
Whereas on $x\in (0,z/2)$, we have $|x-z|\ge \frac z2 $ and $\phi_{P} \le 1$. We have   
\begin{equation}\label{large-time-del}
\begin{aligned}
 \int_0^{z/2} &  (\nu t)^{-1/2}   e^{-\frac{|x-z|^2}{M\nu t}}  
  \delta^{-1} \phi_{P} (\delta^{-1} x) \; dx 
\\
&\le C_0 e^{-\frac{|z|^2}{8M\nu t}}  \delta^{-1}  \int_0^{z/2}   (\nu t)^{-1/2}   e^{-\frac{|x-z|^2}{2M \nu t}}  \; dx 
\\&\le C_0 e^{-\frac{|z|^2}{8M\nu t}}  \delta^{-1} .
  \end{aligned}\end{equation}

It remains to prove that 
\begin{equation}\label{X-bd}  
e^{-\frac{|z|^2}{8M\nu t}}  \delta^{-1} \le C_0 \delta^{-1} e^{8M\nu t}  e^{-z/\sqrt \nu}
\end{equation}
for some constant $C_0$. Indeed, the inequality is clear, for $|z| \ge 8M\nu^{3/2} t$, since 
$ e^{-\frac{|z|^2}{8M\nu t}} \le  e^{-z/\sqrt\nu }  $. Next, for $|z| \le 8M\nu^{3/2} t$, we have 
$$ 
1\le e^{8M \nu t} e^{-z/\sqrt \nu}.
$$ 
This proves \eqref{X-bd}, and so \eqref{large-time-del} is again bounded by 
$C_0  e^{8M\nu t}  \delta^{-1} \phi_{P} (\delta^{-1} z)$, upon recalling that the boundary layer thickness is of order $\delta = \gamma \sqrt \nu$. 

\end{proof}

\begin{lemma}\label{lem-convolutionG21} 
Let 
$$
R(t,x,z): = \nu^{-1/2} e^{-\theta_{\tau}\nu^{-1/2}|x-z|} 
$$ 
for some positive $\theta_\tau$. Then, for any positive $\beta$, there is a constant $C_\tau$, depending on $\theta_\tau$, so that 
$$
\Big \| \int_0^\infty R(t,x,\cdot) \omega_\alpha(x)\; dx \Big\|_{ \beta, \gamma} \le C_\tau \| \omega_\alpha\|_{ \beta, \gamma}.
$$
\end{lemma}

\begin{proof} We need to bound the integral 
$$
\begin{aligned}
\int_0^\infty \nu^{-1/2} e^{-\theta_{\tau}\nu^{-1/2}|x-z|} \Bigl( 1 + \delta^{-1} \phi_{P} (\delta^{-1} x)  \Bigr) e^{-\beta |x|}\; dx.
\end{aligned}
$$
First, taking $\nu$ smaller, if needed, we can assume that $\frac12\theta_\tau \nu^{-1/2}\ge \beta$, and thus by the triangle inequality $|z|\le |x| + |x-z|$, we have 
 $$ e^{-\frac12\theta_\tau \nu^{-1/2}|x-z|} e^{-\beta |x|} \le e^{-\beta |z|}.$$
Next, similarly as done in the previous lemma, we have 
$$
\begin{aligned}
 \int_{z/2}^\infty & \nu^{-1/2} e^{-\theta_{\tau}\nu^{-1/2}|x-z|} 
  \delta^{-1} \phi_{P} (\delta^{-1} x) \; dx 
 \\ &\le C_0   \delta^{-1} \phi_{P} (\delta^{-1} z)  \int_{z/2}^\infty \nu^{-1/2} e^{-\theta_{\tau}\nu^{-1/2}|x-z|} \; dx
  \\&\le C_\tau \delta^{-1} \phi_{P} (\delta^{-1} z) ,\end{aligned}$$
and 
$$
\begin{aligned}
 \int_0^{z/2} & \nu^{-1/2} e^{-\theta_{\tau}\nu^{-1/2}|x-z|} 
  \delta^{-1} \phi_{P} (\delta^{-1} x) \; dx 
  \\& \le  \nu^{-1/2} e^{-\frac12\theta_{\tau}\nu^{-1/2}|z|} 
   \int_0^{z/2}\delta^{-1} \phi_{P} (\delta^{-1} x) \; dx 
\\& \le C_0  \nu^{-1/2} e^{-\frac12\theta_{\tau}\nu^{-1/2}|z|} ,\end{aligned} 
  $$
which is again bounded by $C_\tau \delta^{-1} \phi_{P} (\delta^{-1} z) $, upon recalling that $\delta = \gamma \sqrt \nu$ and $\phi_P(Z) = (1+Z^P)^{-1}$. 
  \end{proof}

\begin{proof}[Proof of Proposition \ref{prop-S2a}] 
In view of \eqref{def-S1212} and \eqref{def-tGreen2}, we have 
$$ 
S_{\alpha,2} \omega_\alpha (z) = \int_0^\infty G_2(t,x,z) \omega_\alpha(x)\; dx.
$$
For each fixed positive $\tau$, we first show that the set of unstable eigenvalues $\lambda_\alpha$ of $L_\alpha$, 
for all $\alpha\in \ZZ^*$, such that $\Re \lambda_\alpha \ge \tau$ is finite. 
We recall that, in the inviscid limit, such eigenvalues are perturbations of eigenvalues of the limit problem $\nu = 0$, namely
of Rayleigh equation.
It is therefore sufficient to study the Rayleigh problem
$$
 \Delta_\alpha \phi - \frac{U''}{U-c}\phi =0, \qquad c = -\frac{\lambda_\alpha}{i\alpha},
 $$
with the boundary condition $\phi_{\vert_{z=0}} =0$. For each $\alpha \in \ZZ^*$, 
it is clear that there are only finitely many unstable eigenvalues, since the Rayleigh operator 
is a compact perturbation of the Laplacian $\Delta_\alpha$. In addition, multiplying the Rayleigh equation by $\overline{\phi}$ 
and integrating by parts, we get   
$$ 
\int_0^\infty (|\partial_z\phi|^2 + \alpha^2 |\phi|^2) \; dz \le \frac{1}{|\Im c|} \int_0^\infty |U''| |\phi|^2\; dz.
$$ 
In particular, if $\alpha^2 |\Im c| \ge \|U''\|_{L^\infty}$, there is no nontrivial solution to the Rayleigh problem. 
This implies that there is no unstable eigenvalue $\lambda_\alpha$, whenever $\alpha \Re \lambda_\alpha \ge \|U''\|_{L^\infty}$. 
In particular, there are no unstable eigenvalues $\Re \lambda_\alpha \ge \tau$, whenever $\alpha\ge \tau^{-1}\|U''\|_{L^\infty}.$ 

This proves that there are finitely many unstable eigenvalues of $L_\alpha$ so that $\Re \lambda_\alpha \ge \tau$ for all $\alpha \in \ZZ^*$.
 In particular, the summation in \eqref{ptw-G22} from Lemma \ref{lem-tGreen} is finite, independent of $\alpha\in \ZZ^*$, 
 yielding 
$$ 
|G_2(t,x,z) |\le C_\tau (\nu t)^{-1/2}e^{\tau t} e^{-\frac{|x-z|^2}{8 \nu t}}  
+  C_\tau e^{(\Re \lambda_0 + \tau)t}\nu^{-1/2} e^{-\theta_{\tau}\nu^{-1/2}|x-z|} 
$$
with $\theta_{\tau} = \frac12 \sqrt{\Re \lambda_0 + \tau + \alpha^2 \nu}$, 
where $\lambda_0$ denotes the maximal unstable eigenvalue.
Finally, applying Lemmas \ref{lem-convolutionG2} and \ref{lem-convolutionG21}, 
respectively, to the above pointwise bounds, we complete the proof of the Proposition.
\end{proof}

%%%%%%%%%%%%%%%%

\section{Elliptic estimates}\label{sec-elliptic}

%%%%%%%%%%%%%%%%

For the sake of completeness we now detail how the weighted estimates on vorticity may be translated into 
weighted estimates on the velocity field.

%%%%%%%%%%%%%%

\subsection{Inverse of Laplace operator in one space dimension}\label{sec-Lap}

%%%%%%%%%%%%%%

Let us now solve the classical Laplace equation
\beq \label{Lap1}
\Delta_\alpha \phi = \partial_z^2 \phi - \alpha^2 \phi = f
\eeq
on the half line $z \ge 0$, with the Dirichlet boundary condition 
\beq \label{Lap2}
\phi(0) = 0 .
\eeq
Note that $\alpha$ is a non zero integer. We assume that $\alpha > 0$.
We start with bounds in the space ${\cal A}^{\beta}$, namely without boundary layers. We will prove
\begin{prop}
If $f \in {\cal A}^{\beta}$, then the solution $\phi$ to \eqref{Lap1}-\eqref{Lap2} belongs to ${\cal A}^{\beta}$ provided  $\beta \le  1/2$. 
In addition, there holds
\beq \label{Lap3}
\alpha^2 \| \phi \|_{\beta} + \alpha  \| \partial_z \phi \|_{\beta} 
+ \| \partial_z^2 \phi \|_{\beta}   \le C \| f \|_{\beta},
\eeq
where the constant $C$ is independent of $\alpha \in \NN^*$.
\end{prop}

\begin{proof} The solution $\phi$ of (\ref{Lap1})-\eqref{Lap2} is explicitly given by
\begin{equation}\label{laplacephi1}
\phi(z) = \int_0^\infty G(x,z) f(x) dx 
\end{equation}
where $G(x,z)$ is the Green function of $\partial_z^2 - \alpha^2$, with the Dirichlet boundary condition. Precisely, we have 
$$
G(x,z) = - {1 \over 2 \alpha } \Bigl( e^{- \alpha | z - x |} - e^{- \alpha | z + x |} \Bigr) .
$$
In particular, $|G(x,z)|\le \alpha^{-1} e^{-\alpha |x-z|}$. Therefore, as $|f(z)|\le \|f\|_\beta e^{-\beta z}$, we have 
$$
| \phi (z) | \le \alpha^{-1} \| f \|_{\beta} \int_0^\infty
e^{- \alpha |   z -   x |} e^{-\beta   x} 
 d  x .
$$
Using the triangle inequality $|z|\le |x|+|x-z|$ and the assumption that $\beta\le 1/2<\alpha$, we have 
$$
| \phi (z) | \le \alpha^{-1} \| f \|_{\beta} e^{-\beta z}\int_0^\infty
e^{-\frac12 \alpha |   z -   x |} 
 d  x \le 2 \alpha^{-2} \| f \|_{\beta} e^{-\beta z},
$$
which yields the claimed bound for $\alpha^2\phi$. 
The estimate for $\partial_z\phi$ follows similarly, upon noting that $|\partial_z G(x,z)|\le e^{-\alpha|x-z|}$. 
Finally, writing $\partial_z^2 \phi = \alpha^2 \phi + f$, we obtain the estimate for $\partial_z^2 \phi$ from that of $\alpha^2 \phi$. 
\end{proof}

Next, we establish similar elliptic estimates when the source term  $f$ has a boundary layer behavior. 

\begin{prop} \label{proplaplace3}
If $f \in {\cal B}^{\beta,\gamma}$, then the solution $\phi$ to \eqref{Lap1}-\eqref{Lap2} belongs to ${\cal A}^{\beta}$ 
provided  $\beta \le  1/2$. In addition, there hold
\beq \label{Lap4}
\begin{aligned} \alpha  \| \phi \|_{\beta} 
+  \| \partial_z \phi \|_{\beta}   &\le C \| f \|_{\beta,\gamma},
\\
\| \partial_z^2 \phi \|_{\beta,\gamma} &\le C \| f \|_{\beta,\gamma} + C \| \alpha f \|_{\beta,\gamma},
\end{aligned}\eeq
where the constant $C$ is independent of $\alpha\in \NN^*$.
\end{prop}

\begin{proof} For a boundary layer function $f \in \cB^{\beta,\gamma}$, using (\ref{laplacephi1}), we have 
$$
\begin{aligned}
| \phi (z) | 
&\le \alpha^{-1} \| f \|_{\beta,\gamma} \int_0^\infty
e^{- \alpha |   z -   x |} e^{-\beta   x} 
\Bigl( 1 + \delta^{-1} \phi_P(\delta^{-1} x) \Bigr) d  x
\\
&\le \alpha^{-1} \| f \|_{\beta,\gamma} e^{-\beta z}\int_0^\infty
e^{- \frac12\alpha |   z -   x |} 
\Bigl( 1 + \delta^{-1} \phi_P(\delta^{-1} x) \Bigr) d  x
\\&\le \alpha^{-1} \| f \|_{\beta,\gamma} e^{-\beta z}\Big( 2\alpha^{-1} + C \Big)
,\end{aligned}$$
upon noting that $\delta^{-1} \phi_P(\delta^{-1} \cdot)$ is bounded in $L^1$. 
This proves the claimed bound for $\alpha \phi$. The estimate for $\partial_z\phi$ follows similarly,
 upon noting that $|\partial_z G(x,z)|\le e^{-\alpha|x-z|}$. 
 The estimate for $\partial_z^2\phi$ follows from $\partial_z^2 \phi = \alpha^2 \phi + f$. 
\end{proof}

Because of the boundary layer behavior of $f$, we cannot get a good control on second order derivatives
 without  an extra control on $f$ (precisely, on $\alpha f$) as is the case in Proposition \ref{proplaplace3}.

%%%%%%%%%%%%%%%%

\subsection{Stream function and vorticity}

%%%%%%%%%%%%%%%%

We recall that the stream function and vorticity are linked through the elliptic equation 
\begin{equation}\label{Lap5}
- \Delta \phi = \omega
\end{equation}
with the zero Dirichlet boundary condition on $\phi$ at $z=0$. Let 
\begin{equation}\label{def-vphi}
v = \nabla^\perp \phi
\end{equation}
be the velocity vector field. We obtain the following Proposition

\begin{prop} \label{inverseLaplace} Let $\beta \le 1/2$ and let $\omega$ be a vorticity function so that 
$\omega_\alpha \in \cB^{\beta,\gamma}$, for $\alpha \in \ZZ$. 
Then, for each $\alpha \in \ZZ$, there hold the following elliptic estimates: 
\beq \label{Lap44}
\begin{aligned}
 \|  v_\alpha \|_{\beta} &\le C \| \omega_\alpha \|_{\beta,\gamma},
 \\
 \| \partial_z v_{2,\alpha}\|_\beta+  \| \partial_z v_{1,\alpha} \|_{\beta,\gamma} 
 &\le C\Big( \| \omega_\alpha \|_{\beta,\gamma} + \|\alpha \omega_\alpha \|_{\beta,\gamma}\Big),
 \\
\|  \psi(z)^{-1}v_{2,\alpha} \|_{\beta} 
&\le C\Big( \| \omega_\alpha \|_{\beta,\gamma} + \|\alpha \omega_\alpha \|_{\beta,\gamma}\Big),
  \end{aligned}
\eeq
with $\psi(z) = z / (1+z)$
\end{prop}
\begin{proof}
In Fourier variables, the elliptic equation \eqref{Lap5} reads 
$$
(\partial_z^2 - \alpha^2) \phi_\alpha = \omega_\alpha.
$$ 
The estimates on $v_\alpha$ and $\partial_z v_\alpha$ are derived using  Proposition \ref{proplaplace3}, 
with $v_{\alpha,1} = \partial_z \phi_\alpha$ and $v_{2,\alpha} = -i\alpha \phi_\alpha$. Next, using $\phi_\alpha(0)=0$, we write 
$$
 \psi(z)^{-1} \phi_\alpha(z) = \psi(z)^{-1} \int_0^z \partial_z \phi_\alpha\; dy 
 \le \| \partial_z \phi_\alpha\|_\beta \, \psi(z)^{-1}\int_0^z e^{-\beta y}\; dy 
 $$
which is  bounded by $C\|\partial_z\phi_\alpha\|_\beta$.  
The proposition follows. 
\end{proof}

%%%%%%%%%%%%%%%%%%

\bibliographystyle{abbrv}
%\bibliography{book-ref}

\def\cprime{$'$} \def\cprime{$'$}

\end{document}